\documentclass[preprint,sort&compress]{elsarticle}
\usepackage{amsmath}
\usepackage{amssymb}
\usepackage{amsthm}
\usepackage{xypic}
\usepackage{graphicx}
\usepackage{mathrsfs}
\usepackage{ifpdf}
\usepackage{paralist}
\allowdisplaybreaks[4]

\numberwithin{equation}{section}

\usepackage{color}

\ifx\red\undefined
\newcommand{\red}{\color{red}}

\renewcommand{\le}{\leqslant}
\renewcommand{\ge}{\geqslant}

\newcommand{\G}{\Gamma}

\newcommand{\e}{\varepsilon}

\renewcommand{\o}{\omega}

\renewcommand{\O}{\Omega(t)}

\renewcommand{\d}{\delta}
\newcommand{\bR}{{\mathbb R}}
\newcommand{\m}{\hspace{1em}}
\newcommand{\mm}{\hspace{2em}}

\newcommand{\for}{\m\text{ for }}

\newcommand\beaa{\begin{eqnarray*}}
\newcommand\eeaa{\end{eqnarray*}}
\newcommand{\re}[1]{\mbox{$($\ref{#1}$)$}}

\newtheorem{thm}{Theorem}[section]
\newtheorem{lem}[thm]{Lemma}
\newtheorem{cor}[thm]{Corollary}

\newtheorem{rems}[thm]{Remarks}

\usepackage[sc]{mathpazo}
\journal{Nonlinear Anal. Real World Appl.}

\begin{document}

\begin{frontmatter}
\title{The existence and linear stability of periodic solution
 for a free boundary problem modeling
tumor growth with a periodic supply of external nutrients}

\author[sysu]{Wenhua He}
\ead{hewh27@mail2.sysu.edu.cn}
\author[sysu]{Ruixiang Xing\corref{cor}}
\ead{xingrx@mail.sysu.edu.cn}

\cortext[cor]{Corresponding author}

\address[sysu]{School of Mathematics, Sun Yat-sen University, Guangzhou 510275, China}

\begin{abstract}
We study a free boundary problem modeling tumor growth with a T-periodic supply $\Phi(t)$ of external nutrients. 
The model contains two parameters $\mu$ and $\widetilde{\sigma}$. We first show that (i) zero radially symmetric solution is globally stable if and only if $\widetilde{\sigma}\ge \frac{1}{T} \int_{0}^{T} \Phi(t) d t$; (ii) If $\widetilde{\sigma}<\frac{1}{T} \int_{0}^{T} \Phi(t) d t$, then there exists a unique radially symmetric positive solution $\left(\sigma_{*}(r, t), p_{*}(r, t), R_{*}(t)\right)$ with period $T$ and it is a global attractor of all positive radially symmetric solutions for all $\mu>0$.
These results are a perfect answer to open problems in Bai and Xu [Pac. J. Appl. Math. 2013(5), 217-223]. Then, considering non-radially symmetric perturbations, we prove that there exists a constant $\mu_{\ast}>0$ such that $\left(\sigma_{*}(r, t), p_{*}(r, t), R_{*}(t)\right)$ is linearly stable for $\mu<\mu_{\ast}$ and linearly unstable for $\mu>\mu_{\ast}$.
\end{abstract}

\end{frontmatter}

\section{Introduction}\label{sec:intro}

Consider a free boundary problem modeling tumor growth with a periodic supply of external nutrients:
\begin{align}\label{1.1}
&\Delta \sigma=\sigma   && x\in \Omega(t),  t>0,\\\label{1.2}
&-\Delta p=\mu(\sigma-\widetilde{\sigma})  && x\in \Omega(t),   t>0,\\\label{1.3}
&V_{n}=-\frac{\partial p}{\partial \mathbf{n}}  && x\in \partial \Omega(t),  t>0,\\\label{1.4}
&\sigma=\Phi(t)     && x\in \partial \Omega(t),    t>0,\\\label{1.5}
&p=\gamma \kappa  &&x\in \partial \Omega(t),    t>0,\\\label{1.6}
&\Omega(0)=\Omega_{0},
\end{align}
where $\O\subseteq \mathbb{R}^{3}$ is the domain occupied by tumor at time $t$, $\sigma$ denotes the concentration of nutrients, $p$ is the pressure in the tumor, $\widetilde{\sigma}$ denotes a threshold concentration for proliferation, $\mu$ is proportional coefficient in Darcy law, $V_{n}$ denotes the velocity of the free boundary in the unit outward normal direction $\mathbf{n}$, $\gamma$ is cell adhesiveness coefficient, $\kappa$ is the mean curvature and $\Phi(t)$ is concentration of external nutrients which is a continuous, positive periodic function satisfying $\Phi(t+T)=\Phi(t)$. The detailed introduction of this model is referred to \cite{Bai2013,Huang2019}.

At first, we study radially symmetric solutions.
For radially symmetric solutions, problem \eqref{1.1}--\eqref{1.6} is reduced to
\begin{align}\label{1.7}
&\frac{1}{r^{2}} \frac{\partial}{\partial r}\left(r^{2} \frac{\partial \sigma}{\partial r}\right)=\sigma    && r \in(0, R(t)),  t>0,\\\label{1.8}
&-\frac{1}{r^{2}} \frac{\partial}{\partial r}\left(r^{2} \frac{\partial p}{\partial r}\right)=\mu(\sigma-\widetilde{\sigma})  && r \in(0, R(t)),  t>0,\\\label{1.9}
&\frac{d R}{d t}(t)=-\frac{\partial p}{\partial r}  && r=R(t),  t>0,\\\label{1.10}
&\sigma=\Phi(t)    && r=R(t),    t>0,\\\label{1.11}
&p= \frac{\gamma}{R(t)}   &&r=R(t),    t>0,\\\label{1.12}
&R(0)=R_{0}.
\end{align}

The solution of \re{1.7} and \re{1.10} is
\begin{align}\label{1.13}
\sigma(r, t)=\Phi(t) \frac{R(t)}{\sinh R(t)} \frac{\sinh r}{r},
\end{align}
and the solution of \re{1.8} and \re{1.11} is
\begin{equation}\label{1.14}
p(r,t)=\frac{1}{6} \mu \widetilde{\sigma} r^{2}-\mu \sigma(r, t)+\frac{\gamma}{R(t)}-\frac{1}{6} \mu \widetilde{\sigma} R^{2}(t)+\mu \Phi(t).
\end{equation}
From \re{1.13} and \re{1.14}, problem \re{1.7}--\re{1.12} is reduced to the following problem:
\begin{align}\label{1.15}
&\frac{d R}{d t}(t)=\mu R(t)\left[\Phi(t) P_{0}(R(t))-\frac{\widetilde{\sigma}}{3}\right],\\ \label{1.155}
&R(0)=R_{0},
\end{align}
where $P_{0}(x)=\frac{x \operatorname{coth} x-1}{x^{2}}$. If $R(t)$ is a solution of \re{1.15}--\re{1.155},  then $(\sigma,$ $ p, R)$  ($\sigma$ and $ p$ are given by \re{1.13} and \re{1.14}, respectively) is a solution of \re{1.7}--\re{1.12}.

Denote
\begin{align}\nonumber
\overline{\Phi}=\frac{1}{T} \int_{0}^{T} \Phi(t) d t, \quad \Phi^{*}=\max _{0 \le t \le T} \Phi(t), \quad \Phi_{*}=\min _{0 \le t \le T} \Phi(t).
\end{align}

Bai and Xu (\cite{Bai2013}) studied problem \re{1.15}--\re{1.155} and proved the following results: \\
$(i)$ If $\widetilde{\sigma}>\overline{\Phi}$, then the solution $R(t)\equiv0$ of  \re{1.15} is globally stable. If the solution $R(t)\equiv0$ of  \re{1.15} is globally stable, then $\widetilde{\sigma}\ge\overline{\Phi}$;\\
$(ii)$ If $\widetilde{\sigma}<\Phi_{\ast}$, then \re{1.15} admits a unique T-periodic solution $R_{\ast}(t)$. Moreover, any positive solution $R(t)$ of \re{1.15} converges to $R_{\ast}(t)$ as $t \rightarrow+\infty$.

Notice that when $\Phi(t)$ isn't a constant function, then $\Phi_{\ast}<\overline{\Phi}$ and there is a gap $\widetilde{\sigma}\in[\Phi_{\ast}, \overline{\Phi})$ between the results of $(i)$ and $(ii)$. So they  proposed two open problems:\\ 
(1) Is $R_{\ast}(t)\equiv0$ globally stable as $\widetilde{\sigma}=\overline{\Phi}$?\\
(2) Does there exist a unique periodic solution $R_{\ast}(t)$ of \re{1.15} as $\widetilde{\sigma}<\overline{\Phi}$? Is it a global attractor of all positive solutions?

We give an affirmative answer to the two open problems. 
The following theorems are our main results.
\begin{thm}\label{thm:1.1}
 When $\widetilde{\sigma}=\overline{\Phi}$, the solution $R(t)\equiv0$ of \re{1.15} is globally stable. 
\end{thm}
Together with \cite[Theorem 2.2 and 2.3]{Bai2013} or the above result (i) , it implies the following result. 
\begin{thm}\label{thm:1.1'}
 The solution $R(t)\equiv0$ of \re{1.15} is globally stable if and only if  $\widetilde{\sigma}\ge\overline{\Phi}$.
\end{thm}


\begin{thm}\label{thm:1.2}
If $\widetilde{\sigma}<\overline{\Phi}$ , then\\
(i) \re{1.15} admits a unique T-periodic positive solution $R_{\ast}(t)$. \\
(ii) For any the positive solution $R(t)$, there exist $\delta>0$ and $C>0$ such that
\begin{align}
|R(t)-R_{\ast}(t)|\le C e^{-\delta t}\qquad for ~    t>0, \ \label{1.17}
\end{align}
i.e., $R(t)-R_{\ast}(t)$ decreases exponentially fast to $0$.
\end{thm}


Theorem \ref{thm:1.1'} and \ref{thm:1.2} give a complete classification for the parameter $\widetilde{\sigma}$ and show us that if the average of the supply $\Phi(t)$ of external nutrients isn't  larger than the threshold concentration $\widetilde{\sigma}$ for proliferation, then all spherical tumors will disappear while  if the average of the supply $\Phi(t)$ of external nutrients is  larger than the threshold concentration $\widetilde{\sigma}$ for proliferation, then there exists a unique spherical tumor with periodic change and all the other spherical tumors don't disappear and they evolve to this periodic tumor.

\vskip 2mm

As a direct corollary of Theorem \ref{thm:1.2}, we have the following result.
\begin{cor}\label{cor:1.3333}
 If $\widetilde{\sigma}<\overline{\Phi}$, then there exists a unique T-periodic solution $(\sigma_{*}(r, t), p_{*}(r, t), R_{*}(t))$ of \re{1.7}--\re{1.12} given by
\begin{align}\nonumber
&\sigma_{\ast}(r, t)=\Phi(t)\frac{ R_{\ast}(t)}{\sinh R_{\ast}(t)} \frac{\sinh r}{r}, \\\label{1.1992}
&p_{\ast}(r, t)=\frac{1}{6} \mu \widetilde{\sigma} r^{2}-\mu\sigma_{\ast}(r, t)+\frac{\gamma}{R_{\ast}(t)}-\frac{1}{6} \mu \widetilde{\sigma} R_{\ast}^{2}(t)+\mu \Phi(t),
\end{align}
where $R_{*}(t)$ is the unique T-periodic positive solution of \re{1.15}.
\end{cor}

Recently, for two-space dimensional problem of \eqref{1.1}--\eqref{1.6}, Huang, Zhang and Hu (\cite{Huang2019}) have studied the linear stability of the periodic solution under all non-radially symmetric perturbations.

In this paper, we also extend the linear stability (\cite[Theorem $1.1$]{Huang2019}) of the periodic solution under non-radially symmetric perturbations from two-space dimensional case to three-space dimensional case. Precisely, considering non-radially symmetric perturbations, we prove that there exists a constant $\mu_{\ast}>0$ such that the periodic solution $\left(\sigma_{*}(r, t), p_{*}(r, t), R_{*}(t)\right)$ (given in Corollary \ref{cor:1.3333}) is linearly stable for $\mu<\mu_{\ast}$ and linearly unstable for $\mu>\mu_{\ast}$ (see Theorem \ref{thm:5.1}).

%

In recent years, many research works have been done on various tumor models (see e.g., \cite{Byrne1997,Friedman1999,Friedman2001,Friedman2001a,Fontelos2003,Friedman2006,Friedman2006c,Cui2007,Friedman2007,Friedman2007a,Zhang2009,Lowengrub2010,Escher2011,Hao2012,Hao2012a,
Wu2012,Wang2014,Wu2015,Wu2017,Friedman2008,Wang2017,Huang2017,Li2017,Pan2018a} and the references therein).
If $\Phi(t)$ is a constant and \re{1.1} is replaced by $\varepsilon\sigma_t-\Delta \sigma+\sigma =0$, many interesting results about the existence and stability of the stationary solution have been established (see \cite{Fontelos2003,Friedman2006,Friedman2006c,Friedman2008}). The tumor model with the general consumption rate of the nutrients and the general tumor cell proliferation rate has been studied by Cui and Escher (see \cite{Cui2007}). If the boundary \eqref{1.4} is replaced by the boundary condition $\frac{\partial \sigma}{\partial n}+ \alpha (\sigma-\overline{\sigma})=0(\alpha>0)$, this model and the general case have been considered by Huang, Zhang and Hu (\cite{Huang2017}) and Cui and Zhuang (\cite{Zhuang2018,Cui2018,Zhuang2018a}). The results about the existence and stability of the stationary solution have been extended in \cite{Zhou2015,Wu2016,Wu2017,Wu2017a} to the case involving Gibbs-Thomson relation
on the boundary. The tumor model in the presence of inhibitor has been investigated in \cite{Wang2014,Wu2015,Li2017,Cui2000a,Cui2002,Wu2007}.
The tumor model with ECM and MDE interactions was
analyzed in \cite{Pan2018a,Li2019}. Friedman et al. (\cite{Friedman2007a,Friedman2007b,Friedman2006a,Wu2009}) have studied the various tumor models in fluid-like
tissue.

The paper is organized as follows. In Section 2, we establish Theorem \ref{thm:1.1} about the global stability of zero solution of \re{1.15}. In Section 3, we derive Theorem \ref{thm:1.2} about the existence and asymptotically stable of the periodic solution under radially symmetric perturbations.
We discuss the linear stability of the periodic solution of three-space dimensional problem under non-radially symmetric perturbations in Section 4.

\section{Stability of zero Equilibrium
}\label{result}
In this section, we shall prove Theorem \ref{thm:1.1}.

\vskip 3mm\noindent
{\it\bf Proof of Theorem \ref{thm:1.1}.} 
Assume $\widetilde{\sigma}=\overline{\Phi}$ and $R(t)$ is a solution of \re{1.15} with the initial value $R_{0}>0$. Then $R(t)>0$.
We split the proof into three steps.

{\bf Step 1:}  At first, we claim
\begin{align}
&R(t+T)\le R(t) \qquad for ~t>0, \label{3.33}\\
&R(t)\le R(a)e^{ \displaystyle\mu\frac{\Phi^{*}-\widetilde{\sigma}}{3}T}  \qquad for ~  t\in[a,a+T], \ \label{3.2}
\end{align}
where $a\ge 0$.

Since $0<P_{0}(x)<\frac{1}{3}$ (\cite{Friedman2006}), \re{1.15} implies
\begin{align}
 \frac{d R}{d t}\le \mu R(t)\left[\frac{\Phi(t)}{3}-\frac{\widetilde{\sigma}}{3}\right]. \ \label{3.222}
\end{align}
Hence
\begin{equation*}
R(t+T) \leqslant R(t)e^{  \displaystyle\int_{t}^{t+T}\mu\left[\frac{\Phi(t)}{3}-\frac{\widetilde{\sigma}}{3}\right] d t}=R(t)e^{  \displaystyle\mu\left[ \displaystyle \frac{ \overline{\Phi}}{3}T  -\frac{\widetilde{\sigma}}{3} T\right]}=R(t).
\end{equation*}
Then \re{3.33} holds.

From \re{3.222}, we have
\begin{align}\nonumber
 \frac{d R}{d t}\le \mu R(t)\left[\frac{\Phi^{\ast}}{3}-\frac{\widetilde{\sigma}}{3}\right],
\end{align}
which implies
\begin{align}\nonumber
 R(t)\le R(a)e^{ \displaystyle \mu\frac{\Phi^{*}-\widetilde{\sigma}}{3}(t-a)}\le R(a)e^{ \displaystyle \mu\frac{\Phi^{*}-\widetilde{\sigma}}{3}T} \qquad for ~  t\in[a,a+T].
\end{align}
Then \re{3.2} is true.

{\bf Step 2:} We claim
\begin{align}
\liminf _{t \rightarrow+\infty}  R(t)=0. \ \label{3.9}
\end{align}
Assume on the contrary
\begin{align}\nonumber
\liminf _{t \rightarrow+\infty}  R(t)=\alpha>0.
\end{align}
For every $\varepsilon>0$, there exists $M>0$ such that
\begin{align}
R(t)>\alpha-\varepsilon  \qquad for ~    t>M. \ \label{3.188}
\end{align}
From \re{1.15}, the fact that $P_{0}$ is strictly decreasing (\cite{Friedman2006}) implies
\begin{align}\nonumber
 \frac{d R}{d t}= \mu R(t)\left[\Phi(t)P_{0}(R(t))-\frac{\widetilde{\sigma}}{3}\right]\le \mu R(t)\left[\Phi(t)P_{0}(\alpha-\varepsilon)- \frac{\widetilde{\sigma}}{3}\right] \qquad for ~   t>M.
\end{align}
Then
\begin{equation}\label{3.13}
R(t^{\ast}+nT) \le R(t^{\ast})e^{ \displaystyle \int_{t^{\ast}}^{t^{\ast}+nT}\mu\left[\Phi(t)P_{0}(\alpha-\varepsilon)-\frac{\widetilde{\sigma}}{3}\right] d t}=R(t^{\ast})e^{ \displaystyle\mu nT \left[P_{0}(\alpha-\varepsilon)\overline{\Phi}-\frac{\widetilde{\sigma}}{3}\right]},
\end{equation}
where $n\geqslant1$ is an integer and $t^{\ast}>M$.
Notice
\begin{align}\nonumber
P_{0}(\alpha-\varepsilon)\overline{\Phi}-\frac{\widetilde{\sigma}}{3}<\frac{1}{3}\overline{\Phi}-\frac{\widetilde{\sigma}}{3}=0.
\end{align}
Letting $n\rightarrow\infty$ in \eqref{3.13}, we have
\begin{align}\nonumber
 R(t^{\ast}+nT) \rightarrow0.
\end{align}
It contracts with \re{3.188}. Hence \re{3.9} holds.

{\bf Step 3:} We shall prove
\begin{align}
\lim _{t \rightarrow+\infty} R(t)=0. \ \label{3.16}
\end{align}
From \re{3.9}, for all $\varepsilon>0$, there exist a sequence $t_{n}\rightarrow\infty$ and $M>0$ such that
\begin{align}\nonumber
R(t_{n})<\varepsilon \qquad for ~    t_{n}>M.
\end{align}
\re{3.33} implies
\begin{align}\nonumber
R(t_{N}+kT)\le R(t_{N})< \varepsilon,
\end{align}
where $t_{N}>M$ and $k\geqslant1$ is an integer. For $t>t_{N}$, there exists $k_{0}$ such that $t\in[ t_{N}+k_{0}T,t_{N}+(k_{0}+1)T)$. Together with \re{3.2}, we obtain
\begin{align}\begin{array}{ll}\nonumber
R(t) \le R(t_{N}+k_{0}T)e^{ \displaystyle\mu \frac{\Phi^{*}-\widetilde{\sigma}}{3}T} \le \varepsilon  \displaystyle e^{  \displaystyle\mu \frac{\Phi^{*}-\widetilde{\sigma}}{3} T}.
\end{array}
\end{align}
Then \re{3.16} is true, which completes the proof.
\hfill$\square$

\begin{rems}
If $\widetilde{\sigma}>\overline{\Phi}$, the key step (\cite[$(2.4)$]{Bai2013}) tells
\begin{equation*}
R(\xi+nT) \leqslant R(\xi)e^{  \frac{nT}{3}\mu(\overline{\Phi}-\widetilde{\sigma})}\rightarrow 0, \qquad  n\rightarrow\infty.
\end{equation*}
Then 
$R(t)\equiv0$ is globally stable.
If $\widetilde{\sigma}=\overline{\Phi}$, $(2.4)$ in \cite{Bai2013} only shows
\begin{equation*}
R(\xi+nT) \leqslant R(\xi).
\end{equation*}
By this method, one can not get $R(t)$ converges to $0$ as $t\rightarrow\infty$.

When $\widetilde{\sigma}=\overline{\Phi}$, we find good properties \re{3.33} and \re{3.2}, i.e., $R(t+nT)$ is a decreasing function in $n$ and in one period, $R(t)(t\in[a,a+T])$ can be controlled by
$CR(a)$. The two properties and $\liminf\limits_{t \rightarrow+\infty}  R(t)=0$ (\re{3.9} in Step 2) ensure Theorem \ref{1.1} holds.
\end{rems}

\section{Existence, Uniqueness and Stability of the Periodic Solution
}\label{result}
In this section, we shall prove Theorem \ref{thm:1.2}.


\vskip 3mm\noindent
{\it\bf Proof of Theorem \ref{thm:1.2}.}
The facts that $\widetilde{\sigma}<\overline{\Phi}\le\Phi^{\ast}$ , $0<P_{0}(x)<\frac{1}{3}$ and $P_{0}(x)$ is strictly decreasing imply that $x_{2}=P_{0}^{-1}(\frac{\widetilde{\sigma}}{3\Phi^{\ast}})$ and $\overline{x}=\frac{P_{0}^{-1}(\frac{\widetilde{\sigma}}{3\overline{\Phi}})}{e^{\mu\frac{\Phi^{*}-\widetilde{\sigma}}{3}T}}$ are well defined. Since $P_{0}(x)$ is strictly decreasing, it follows
\begin{align}\nonumber
\overline{x}<x_{2}.
\end{align}
For each $R_{0}\in[\overline{x}, x_{2}]$, we let $R(t)$ be the solution of \re{1.15} with the initial value $R(0)=R_{0}$. Define the map $F$: $[\overline{x}, x_{2}]\rightarrow \mathbb{R}$ by
 $$F(R_{0})=R(T).$$
At first, we show that $F$ maps $[\overline{x}, x_{2}]$ into itself.

Since $x_{2}$ is a upper solution of the \re{1.15} and $R(0)\le x_{2}$, the comparison theorem implies
\begin{align}\nonumber
R(t)\le x_{2}  \qquad for ~t>0.
\end{align}
Then
\begin{align}
R(T)\le x_{2}. \ \label{4.3}
\end{align}
On the other hand, we define $\overline{R}$ by
\begin{align}\left\{\begin{array}{ll}
  \dfrac{d \overline{R}}{d t}=\mu \overline{R}(t)\left[\Phi(t) P_{0}(\overline{R}(t))-\displaystyle\frac{\widetilde{\sigma}}{3}\right],\\
  \overline{R}(0)=\overline{x}.
\end{array}\right.
\label{4.5}
\end{align}
By comparison theorem, we obtain
\begin{align}
R(t)\ge \overline{R}(t) \qquad for ~ t>0. \ \label{4.6}
\end{align}
The fact that $0<P_{0}(x)<\frac{1}{3}$ implies
\begin{align}\nonumber
  \frac{d \overline{R}}{d t}=\mu \overline{R}(t)\left[\Phi(t) P_{0}(\overline{R}(t))-\frac{\widetilde{\sigma}}{3}\right]\le\mu \overline{R}(t)\left[ \frac{\Phi^{\ast}}{3}- \frac{\widetilde{\sigma}}{3}\right].
\end{align}
Then
\begin{equation}\nonumber
\overline{R}(t) \leqslant \overline{x} e^{ \displaystyle\mu\frac{\Phi^{*}-\widetilde{\sigma}}{3} t } \leqslant \overline{x} e^{ \displaystyle\mu\frac{\Phi^{*} - \widetilde{\sigma}}{3} T}=P_{0}^{-1}\left(\frac{\widetilde{\sigma}}{3 \overline{\Phi}}\right) \qquad for ~t\in[0,T].
\end{equation}
Since $P_{0}(x)$ is strictly decreasing, it follows
\begin{align}\nonumber
P_{0}(\overline{R}(t))\geqslant \frac{\widetilde{\sigma}}{3\overline{\Phi}} \qquad for~ t\in[0, T].
\end{align}
Together with the first equation of \re{4.5}, we get
\begin{align*}
  \frac{d \overline{R}}{d t}\ge  \mu \overline{R}(t)\left[ \Phi(t)\frac{\widetilde{\sigma}}{3\overline{\Phi}}- \frac{\widetilde{\sigma}}{3}\right] \qquad for ~ t\in(0, T).
\end{align*}
Hence
\begin{align}
\overline{R}(T)\ge \overline{R}(0)e^{\displaystyle \int_{0}^{T}\mu\left[\Phi(t)\frac{\widetilde{\sigma}}{3\overline{\Phi}}-\frac{\widetilde{\sigma}}{3}\right]dt}=\overline{R}(0)=\overline{x}.
\label{4.10}
\end{align}
From \re{4.3}, \re{4.6} and \re{4.10}, we get
$$R(T)\in[\overline{x}, x_{2}].$$
Then $F$ maps $[\overline{x}, x_{2}]$ into itself. Since the solution $R$ depends continuously on the initial value $R_{0}$, it follows that $F$ is continuous. Brouwer's fixed point theorem implies that $F$ has a fixed point $R_{\ast}(0)$. Then the solution $R_{\ast}(t)$ of \re{1.15} with the initial value $R_{\ast}(0)$ is a positive T-periodic solution. So far, we have shown the existence of the periodic solution. The uniqueness of the periodic solution will be given at the end of the proof.

Let
\begin{align}\label{4.7}
R_{min}=\min\limits_{t>0}\left\{R_{\ast}(t)\right\}\qquad \text{and} \qquad R_{max}=\max\limits_{t>0}\left\{R_{\ast}(t)\right\}.
\end{align}
The uniqueness of the solution to the initial value problem implies that $R_{min}>0$ and $R_{max}>0$.

Next we turn to  prove $(ii)$.
Assume that $R(t)$ is the solution of \re{1.15} with the initial value $R(0)>0$.

Let
\begin{align}\nonumber
R(t)=R_{\ast}(t) e^{y(t)}.
\end{align}
Then $y$ satisfies the following equation
\begin{align}\begin{array}{ll}
y'(t) =  \mu\Phi(t)[P_{0}(R_{\ast}(t)e^{y(t)})-P_{0}(R_{\ast}(t))].
\end{array}
\label{4.12}
\end{align}
To prove \re{1.17}, it is sufficient to show that there exist $\delta>0$ and $C>0$ such that
\begin{align}\nonumber
|e^{y(t)}-1|\le C e^{-\delta t}\qquad for ~    t>0.
\end{align}
The uniqueness of the solution to the initial value problem implies that if $R(0)>R _{*}(0)$, then $R(t)>R_{*}(t)$ and if $R(0)<R _{*}(0)$, then $R(t)<R_{*}(t)$. Then $y(t)>0$ if  $y(0)>0$ and $y(t)<0$ if  $y(t)<0$. Hence the arguments  are divided into two cases according to the sign of $y(t)$.

Case A: $y(t)>0$.

From \re{4.12}  and the mean value theorem, we get
\begin{align}\nonumber
y'(t) e^{y(t)}&=  \mu\Phi(t)P_0'(\zeta(t))R_{\ast}(t)(e^{y(t)}-1)e^{y(t)}   \\\nonumber
&\le -\mu\Phi_{\ast}M_{min} R_{min}   (e^{y(t)}-1)\nonumber,
\end{align}
where $\zeta(t)\in[R_{\ast}(t), R_{\ast}(t)e^{y(t)}]\subseteq [R_{min}, R_{max}e^{y(0)}]$ 
and $M_{min}=\min\limits_{x\in [R_{min}, R_{max}e^{y(0)}]}\left\{-P_0'(x)\right\}>0$.

Hence
\begin{align}
\frac{(e^{y(t)}-1)'}{(e^{y(t)}-1)}\le -\mu\Phi_{\ast}M_{min} R_{min} \label{4.1511}.
\end{align}
Integrating \re{4.1511} over $[0,t]$, we have
\begin{align}
e^{y(t)}-1\le (e^{y(0)}-1)e^{-\mu\Phi_{\ast}M_{min} R_{min} t}\qquad for ~    t>0\label{4.15111}.
\end{align}

Case B: $y(t)<0$.

From \re{4.12} and the fact that $P_{0}$ is strictly decreasing, we get that $y'(t)>0$. Combining the mean value theorem, we have
\begin{align}\nonumber
-y'(t) e^{y(t)}&=  -\mu\Phi(t)P_0'(\eta(t))R_{\ast}(t)(e^{y(t)}-1)e^{y(t)}   \\\nonumber
&\le  -\mu\Phi_{\ast} M_{min}R_{min}   (1-e^{y(t)})e^{y(0)}\nonumber,
\end{align}
where $\eta(t)\in[R_{\ast}(t)e^{y(t)}, R_{\ast}(t)]\subseteq [R_{min}e^{y(0)}, R_{max}]$. 

Hence
\begin{align}
\frac{(1-e^{y(t)})'}{(1-e^{y(t)})}\le -\mu\Phi_{\ast}  M_{min}R_{min}e^{y(0)}\label{4.15122}.
\end{align}
Integrating \re{4.15122} over $[0,t]$, we have
\begin{align}
1-e^{y(t)}\le (1-e^{y(0)})e^{-\mu\Phi_{\ast} M_{min}R_{min} e^{y(0)}t}\qquad for ~    t>0\label{4.1511133}.
\end{align}
Taking $\delta=\min\{\mu\Phi_{\ast} M_{min}R_{min} , \mu\Phi_{\ast}M_{min} R_{min} e^{y(0)}\}$ and $C=|1-e^{y(0)}|$.   \re{4.15111} and \re{4.1511133} imply \re{1.17}.

Finally, we show that the solution $R_{\ast}(t)$ is unique. Otherwise, by \re{1.17}, we obtain

$$
\left|R_{*}^{1}(t)-R_{*}^{2}(t)\right|\leqslant\left|R(t)-R_{*}^{1}(t)\right|+\left|R(t)-R_{*}^{2}(t)\right| \rightarrow 0 \qquad t\rightarrow \infty.
$$
Hence $R_{*}^{1}(t)=R_{*}^{2}(t)$, which completes the proof.\hfill$\square$

\begin{rems}
Bai and Xu (\cite{Bai2013}) applied Brouwer's fixed point theorem to show the existence of periodic solution. They need the condition $\widetilde{\sigma}<\Phi_{\ast}$ to construct sub-solution
$x_{1}=P_{0}^{-1}\left(\frac{\widetilde{\sigma}}{3 \Phi_{\ast}}\right)$. To get rid of the condition $\widetilde{\sigma}<\Phi_{\ast}$, we use a sub-solution $\overline{R}(t)$ (given in  \re{4.5}) to replace $x_{1}$.
\end{rems}

\begin{rems}
The proof of Theorem \ref{thm:1.2} is still valid for two-space dimensional problem. The existence scope $\widetilde{\sigma}\in\left(0, \Phi_{*}\right)$ of periodic solution in \cite[Theorem $2.1$]{Huang2019} can be extended to $\widetilde{\sigma}\in(0, \overline{\Phi})$.
\end{rems}

\section{ Linear stability of the periodic
solution under non-radially symmetric perturbations}\label{result}
In this section, we consider linear stability of the unique radially symmetric T-periodic positive solution $ \left(\sigma_{*}(r, t), p_{*}(r, t), R_{*}(t)\right)$ obtained in Corollary \ref{cor:1.3333} under non-radially symmetric perturbations.

Substituting
\begin{equation}\nonumber
\begin{aligned} \partial \Omega(t) : r &=R_{*}(t)+\varepsilon \rho(\theta, \phi, t)+O(\varepsilon^{2}), \\
\sigma(r, \theta, \phi, t) &=\sigma_{*}(r, t)+\varepsilon w(r, \theta, \phi,  t)+O(\varepsilon^{2}), \\
p(r, \theta, \phi, t) &=p_{*}(r, t)+\varepsilon q(r, \theta, \phi, t)+O(\varepsilon^{2})
\end{aligned}
\end{equation}
into \re{1.1}--\re{1.6} and collecting the $\varepsilon$-order terms, we can get  the linearized system of \re{1.1}--\re{1.6}  at the radially symmetric T-periodic solution $ \left(\sigma_{*}(r, t), p_{*}(r, t), R_{*}(t)\right)$.

\re{1.3} and \cite{Fontelos2003,Friedman2006,Friedman2008} imply
\begin{align}\nonumber
V_{\mathrm{n}}=R_{*}^{\prime}(t)+\varepsilon \rho_{t}+O(\varepsilon^{2}) \ \ and \ \
\kappa =
\frac{1}{R_{*}(t)}-\frac\e{R_{*}^2(t)}\Big(\rho+\frac12\Delta_\o\rho\Big)+O(\e^2).
\end{align}
Then the linearized system has the following form:
\begin{align} \nonumber
&\Delta w= w  &&r\in(0,R_{*}(t)), t>0, \\ \nonumber
&w(R_{*}(t),\theta,\phi,t)= - \frac{\partial \sigma_{*}}{\partial r}(R_{*}(t),t)\rho(\theta,\phi,t)  && t>0,\\\nonumber
&-\Delta q=\mu w  &&r\in(0,R_{*}(t)), t>0,\\\nonumber
&q(R_{*}(t),\theta,\phi,t)=- \frac{\partial p_{*}}{\partial r}(R_{*}(t),t) \rho(\theta, \phi, t)-\; \frac\gamma{R^2_{*}(t)}\Big(\rho+\frac12\Delta_\o\rho\Big) && t>0,\\\nonumber
&\frac{d \rho}{d t}=- \frac{\partial^{2} p_{*}}{\partial r^{2}}(R_{*}(t),t) \rho(\theta, \phi, t)-\frac{\partial q}{\partial r}(R_{*}(t),t)&& t>0.
\end{align}


At first, we give some preliminaries.

The Bessel function (\cite{Watson1995}) is given by
\begin{align}\nonumber
I_n(r) = \Big(\frac r2\Big)^n
\sum_{k=0}^\infty\frac1{k!\G(n+k+1)}\Big(\frac r2\Big)^{2k},
\end{align}
and has the following properties:
\begin{align} \nonumber
&\Big( \frac{\partial^2}{\partial r^2} + \frac2r\frac\partial{\partial
 r}-\frac{n(n+1)}{r^2} \Big)\frac{I_{n+1/2}(r)}{r^{1/2}} =
 \frac{I_{n+1/2}(r)}{r^{1/2}},\\\label{B.9}
&\Big( \frac{d}{dr} - \frac{n}r\Big)
 \frac{I_{n+1/2}(r)}{r^{1/2}} = \frac{I_{n+3/2}(r)}{r^{1/2}}.
 \end{align}

A useful function $P_n$  (\cite{Friedman2006}) is given by
\begin{align}\nonumber
P_n(r) = \frac{I_{n+3/2}(r)}{r I_{n+1/2}(r)}\mm
n=0,1,2,3,\cdots,
\end{align}
and has the following properties:
\begin{align}\nonumber
& P_0(r) =  r^{-1} \coth r- r^{-2},\\\nonumber
&\frac{d P_n}{dr}(r)<0\qquad r>0,\\\label{2.8}
&0<P_n(r)\le\frac{1}{2n+3}\qquad r\ge0,\\\label{eq1}
&P_0(r)=\frac{1}{r^{2}P_1(r)+3},\\
\label{2.10}
& P_n(r)>P_{n+1}(r)\qquad \forall n\geqslant 0,r>0.
 \end{align}


\begin{lem}
 The following relations hold:
\begin{align} \label{2.23d}
&\frac{\partial \sigma_*}{\partial r}(R_*(t),t)=\Phi(t)R_*(t)P_{0} (R_*(t)),\\\label{2.24d}
&\frac{\partial^{2} \sigma_*}{\partial r^{2}}(R_*(t),t)=\Phi(t)[1-2P_{0} (R_*(t))],\\
\label{2.25fd}
&\frac{\partial p_*}{\partial r}(R_*(t),t)=-\frac{d R_*(t)}{d t},\\\label{2.26d}
&\frac{\partial^{2} p_*}{\partial r^{2}}(R_*(t),t)=-\frac{1}{ R_*(t)} \frac{d R_*}{d t}   - \mu\Phi(t)R^2_*(t) P_{0}(R_*(t))P_{1}(R_*(t)) .
\end{align}
\end{lem}
\begin{proof} The proofs of \eqref{2.23d}--\eqref{2.24d} are similar to  \cite[Lemma 2.1]{Friedman2008}, we omit them.
It remains to show that \eqref{2.25fd}--\eqref{2.26d} hold.

\re{1.9} implies \re{2.25fd}.  From \re{1.1992} and \re{2.24d}, we obtain
\begin{align}\nonumber
\frac{\partial^{2} p_*}{\partial r^{2}}(R_*(t),t)=\frac{1}{3} \mu \widetilde{\sigma}- \mu\Phi(t)[1-2P_{0}(R_*(t))].
\end{align}
Together with \re{1.15} and \re{eq1}, we have
\begin{align*}
\frac{\partial^{2} p_*}{\partial r^{2}}(R_*(t),t)
&= - \frac{1}{ R_*(t)} \frac{d R_*}{d t}   - \mu\Phi(t)[1-3P_{0}(R_*(t))]\\
&=- \frac{1}{ R_*(t)} \frac{d R_*}{d t}   - \mu\Phi(t)R^2_*(t) P_{0}(R_*(t))P_{1}(R_*(t))  .
\end{align*}
Then \re{2.26d} is true.
%
\end{proof}

%
%

Let
\begin{align} \nonumber
  \rho(\theta,\phi,t)  &= \sum_{n=0}^{+\infty} \sum_{m=-n}^{n}\rho_{n,m}(t) Y_{n,m}(\theta,\phi),\\\nonumber
  w(r,\theta,\phi,t)& =  \sum_{n=0}^{+\infty} \sum_{m=-n}^{n} w_{n,m}(r,t)Y_{n,m}(\theta,\phi),\\\nonumber
  q(r,\theta,\phi,t) & =   \sum_{n=0}^{+\infty} \sum_{m=-n}^{n} q_{n,m}(r,t)Y_{n,m}(\theta,\phi),
\end{align}
where the spherical harmonic function
\begin{align}\nonumber
 Y_{n,m}(\theta,\phi) = (-1)^m\sqrt{\frac{(2n+1)(n-m)!}{2(n+m)!}}
 \;P^m_n(\cos\theta) \frac{e^{im\phi}}{\sqrt{2\pi}} \m
 (m=-n,\cdots,n)
\end{align}
in $\bR^3$, where
$$
 P^m_n(z) = \frac1{2^nn!} (1-z^2)^{m/2}
 \frac{d^{n+m}}{dz^{n+m}}(z^2-1)^n
$$
is the Legendre polynomial.
$\{ Y_{n,m}\}$ is a complete orthonormal basis for $L^2(\mathbb{S}^2)$. 

Applying the relation (\cite{Friedman2006})
\begin{align}\nonumber
 \Delta_\o  Y_{n,m}(\theta,\phi) + n(n+1)  Y_{n,m}(\theta,\phi) =0,
\end{align}
we have
\begin{align} \label{5.13}
&\frac{\partial^2 w_{n, m}}{\partial r^2}(r, t)+\frac{2}{r} \frac{\partial w_{n, m}}{\partial r} (r, t)-\left(\frac{n(n+1)}{r^{2}}+1\right) w_{n, m}(r, t)=0  &&r\in(0,R_{*}(t)), t>0, \\ \label{5.14}
&w_{n, m}(R_{*}(t),\theta,\phi,t)= - \frac{\partial \sigma_{*}}{\partial r}(R_{*}(t),t)\rho_{n, m}(t)  && t>0,\\ \label{5.15}
& \frac{\partial^2 q_{n, m}}{\partial r^2}(r, t)+\frac{2}{r}  \frac{\partial q_{n, m}}{\partial r} (r, t)-\frac{n(n+1)}{r^{2}} q_{n, m}(r, t)=-\mu w_{n, m}(r, t) &&r\in(0,R_{*}(t)), t>0,\\ \label{5.16}
& q_{n, m}\left(R_{*}(t), t\right)=\left(\frac{n(n+1)}{2}-1\right)\frac{\gamma\rho_{n, m}(t)}{R_{*}^{2}(t)} -\frac{\partial p_{*}}{\partial r}(R_{*}(t),t)\rho_{n, m}(t) && t>0,\\ \label{5.17}
& \frac{d \rho_{n, m}}{d t}(t)= -\frac{\partial^{2} p_{*}}{\partial r^{2}}(R_{*}(t),t) \rho_{n, m}(t)- \frac{\partial q_{n, m}}{\partial r} (R_{*}(t),t)                                      && t>0.
\end{align}
The solution $w_{n,m}$ of \re{5.13}--\re{5.14} is given by
\begin{align}\nonumber
w_{n,m}(r,t) = -\frac{\partial \sigma_{*}}{\partial r}(R_{*}(t),t)
 \frac{R_{*}^{1/2}(t)}{I_{n+1/2}(R_{*}(t))}\frac{I_{n+1/2}(r)}{r^{1/2}}\rho_{n,m}(t).
\end{align}
Define
\begin{align}\nonumber\psi_{n,m}=q_{n,m}+\mu w_{n,m}.\end{align}
From \re{5.13}--\re{5.16}, we obtain
\begin{align*} 
&\frac{\partial^2 \psi_{n, m}}{\partial r^2}(r, t)+\frac{2}{r} \frac{\partial \psi_{n, m}}{\partial r}(r, t)-\frac{n(n+1)}{r^{2}} \psi_{n, m}(r, t)=0  \qquad r\in(0,R_{*}(t)), \\ 
& \psi_{n, m}\left(R_{*}(t), t\right)= \left(\frac{n(n+1)}{2}-1\right)\frac{\gamma\rho_{n, m}(t)}{R_{*}^{2}(t)}-  \frac{\partial p_{*}}{\partial r} (R_{*}(t),t)\rho_{n, m}(t) -\mu\frac{\partial \sigma_{*}}{\partial r}(R_{*}(t),t) \rho_{n, m}(t).
\end{align*}
The solution of the above problem is
\begin{align*}
{\psi_{n, m}\left(r, t\right)= \frac{r^n}{R_{*}^{n}(t)}\Big[\left(\frac{n(n+1)}{2}-1\right)\frac{\gamma\rho_{n, m}(t)}{R_{*}^{2}(t)}- \frac{\partial p_{*}}{\partial r}(R_{*}(t),t) \rho_{n, m}(t)}- \mu\frac{\partial \sigma_{*}}{\partial r}(R_{*}(t),t) \rho_{n, m}(t)\Big].
\end{align*}
Then
\begin{align}\nonumber
q_{n, m}(r, t)=&\frac{r^n}{R_{*}^{n}(t)}\Big[\left(\frac{n(n+1)}{2}-1\right)\frac{\gamma\rho_{n, m}(t)}{R_{*}^{2}(t)}- \frac{\partial p_{*}}{\partial r}(R_{*}(t),t)  \rho_{n, m}(t)- \mu\frac{\partial \sigma_{*}}{\partial r}(R_{*}(t),t)  \rho_{n, m}(t)\Big]\\
&\qquad +\mu \frac{\partial \sigma_{*}}{\partial r}(R_{*}(t),t) \frac{R_{*}^{1/2}(t)}{I_{n+1/2}(R_{*}(t))} \frac{I_{n+1/2}(r)}{r^{1/2}}\rho_{n,m}(t).
\nonumber
\end{align}
Differentiating the above equation in $r$, using \re{B.9} and taking $r=R_{*}(t)$, we obtain
\begin{align}\nonumber
\frac{\partial q_{n,m}}{\partial r}(R_{*}(t),t) =&\frac{n}{R_{*}(t)}\Big[\left(\frac{n(n+1)}{2}-1\right)\frac{\gamma\rho_{n, m}(t)}{R_{*}^{2}(t)}- \frac{\partial p_{*}}{\partial r}(R_{*}(t),t)  \rho_{n, m}(t)- \mu\frac{\partial \sigma_{*}}{\partial r}(R_{*}(t),t)  \rho_{n, m}(t)\Big]\\\nonumber
&\quad+\mu \frac{\partial \sigma_{*}}{\partial r}(R_{*}(t),t) \left[ \frac{n}{R_{*}(t)} + \frac{I_{n+3/2}(R_{*}(t))}{I_{n+1/2}(R_{*}(t))} \right] \rho_{n,m}(t)\\\nonumber
 =&\frac{n}{R_{*}(t)}\Big[\left(\frac{n(n+1)}{2}-1\right)\frac{\gamma\rho_{n, m}(t)}{R_{*}^{2}(t)}- \frac{\partial p_{*}}{\partial r}(R_{*}(t),t)  \rho_{n, m}(t)\Big]\\\nonumber
&\quad+\mu \frac{\partial \sigma_{*}}{\partial r}(R_{*}(t),t) R_{*}(t) P_n(  R_{*}(t))     \rho_{n,m}(t).
\end{align}
Plugging \re{2.23d} and \re{2.25fd} into the above equation, we obtain
\begin{align}\nonumber
&\frac{\partial q_{n,m}}{\partial r}(R_{*}(t),t)\\\label{4.35}
&  \quad=\Big\{
  \frac{n}{R_{*}(t)}\Big[\frac\gamma{R^2_{*}(t)}\Big(\frac{n(n+1)}2-1\Big)+\frac{d R_{*}}{d t}\Big]+ \mu \Phi(t)R^{2}_{*}(t)P_{0} (R_{*}(t))P_{n} (R_{*}(t))\Big\}\rho_{n, m}(t).
\end{align}
Substituting \re{2.26d} and \re{4.35} into \re{5.17}, we have
\begin{align}\nonumber
\frac{d \rho_{n, m}}{d t}\!=\!-\!\left\{\frac{d R_{*}}{d t}\frac{n\!-\!1}{R_{*}(t)}\!+\!\frac{\gamma n}{R^{3}_{*}(t)}\Big(\frac{n(n+1)}{2}\!-\!1\Big)\!-\!\mu \Phi(t) R_{*}^{2}(t) P_{0}(R_{*}(t))[P_{1}(R_{*}(t))\!-\!P_{n}(R_{*}(t))] \right\}\rho_{n, m}(t).    
\end{align}
Hence,
\begin{align}\nonumber
&\rho_{n, m}(t)\\\label{5.26}
&\!=\!\rho_{n, m}(0)\!\exp\! \Big\{\!\!-\!\!\int_{0}^{t}\frac{d R_{*}}{d t}\frac{n\!-\!1}{R_{*}(t)}\!+\!\frac{\gamma n}{R^{3}_{*}(t)}\Big(\frac{n(n+1)}{2}\!-\!1\Big)\!-\!\mu \Phi(t) R_{*}^{2}(t) P_{0}(R_{*}(t))[P_{1}(R_{*}(t))\!-\!P_{n}(R_{*}(t))]dt\Big\}.
\end{align}
At first,  we give an estimate for $\rho_{0}(t)$.
\begin{lem}\label{lem:5.2}
For any $\mu>0$, there exist $\delta>0$ and $M$ such that
\begin{equation}\nonumber
\left|\rho_{0}(t)\right| \leqslant \left|\rho_{0}(0)\right| e^{-\delta t} \qquad \text { for }  t>M.
\end{equation}
\end{lem}
\begin{proof}
Plugging $n=0$ into \re{5.26}, we have
\begin{align}\nonumber
\rho_{0}(t)
&=\rho_{0}(0) \exp \Big\{-\int_{0}^{t}-\frac{d R_{*}(t)}{d t}\frac{1}{R_{*}(t)}+\mu \Phi(t) R_{*}^{2}(t) P_{0}(R_{*}(t))[P_{0}(R_{*}(t))-P_{1}(R_{*}(t))]dt\Big\}\\\nonumber
&=\rho_{0}(0)\frac{R_{*}(t)}{R_{*}(0)}\exp \Big\{-  \int_{0}^{t}  \mu \Phi(t) R_{*}^{2}(t) P_{0}(R_{*}(t))[P_{0}(R_{*}(t))-P_{1}(R_{*}(t))]dt\Big\}.
\end{align}
From \re{2.10}, 
we get
\begin{equation}\label{5.29}
\mu \Phi(t) R_{*}^{2}(t) P_{0}(R_{*}(t))[P_{0}(R_{*}(t))-P_{1}(R_{*}(t))]\geqslant0.
\end{equation}
For any $t > T$, there exist a positive integer $m$ and $\tau \in[0, T)$ such that $t=m T+\tau$. The fact that $R_{*}(t)$ and $\Phi(t) $ are T-periodic and \re{5.29} imply
\begin{align}\nonumber
\left|\rho_{0}(t)\right|
&=\left|\rho_{0}(0)\right|\frac{R_{*}(t)}{R_{*}(0)}\exp\Big\{-(\int_{0}^{mT}+\int_{mT}^{mT+\tau})\mu \Phi(t) R_{*}^{2}(t)P_{0}(R_{*}(t))[P_{0}(R_{*}(t))-P_{1}(R_{*}(t))]dt\Big\}\\\nonumber
&\leqslant\left|\rho_{0}(0)\right|\frac{R_{*}(\tau)}{R_{*}(0)}\exp\Big\{- m\int_{0}^{T}\mu \Phi(t) R_{*}^{2}(t) P_{0}(R_{*}(t))[P_{0}(R_{*}(t))-P_{1}(R_{*}(t))]dt\Big\}\\\nonumber
& \le \left|\rho_{0}(0)\right|\frac{R_{*}(\tau)}{R_{*}(0)}\exp\Big\{-\mu\overline{\Phi}\widetilde{\delta}mT  \Big\}\\\nonumber
& \le \left|\rho_{0}(0)\right|\frac{R_{max}}{R_{min}}\exp\Big\{-\mu\overline{\Phi}\widetilde{\delta}(t-T)  \Big\}\\
& =\left|\rho_{0}(0)\right|\frac{R_{max}}{R_{min}}e^{\mu\overline{\Phi}\widetilde{\delta}T}e^{-\mu\overline{\Phi}\widetilde{\delta}t},
\nonumber
\end{align}
where $R_{min}=\min\limits_{t>0}\left\{R_{\ast}(t)\right\}$, $R_{max}=\max\limits_{t>0}\left\{R_{\ast}(t)\right\}$ and  $\widetilde{\delta}=\min\limits_{x\in [R_{min}, R_{max}]}\{x^{2} P_{0}(x)[P_{0}(x)-P_{1}(x)]\}>0$.

Then there exists $M>0$ such that
\begin{equation}\nonumber
\left|\rho_{0}(t)\right|\leqslant\left|\rho_{0}(0)\right|e^{-\frac{\mu\overline{\Phi}\widetilde{\delta}}{2}t} \qquad \text { for }  t>M.
\end{equation}
Taking $\delta=\frac{\mu\overline{\Phi}\widetilde{\delta}}{2}$, we complete the proof.
\end{proof}

\vskip 2mm

Next, we give an estimate for $\rho_{1}(t)$.
\begin{lem}\label{lem:5.3}
For any $\mu>0$, we have
\begin{equation}\label{5.31}
\rho_{1, m}(t)=\rho_{1, m}(0).
\end{equation}
\end{lem}
\begin{proof}
Substituting $n=1$ into \re{5.26}, we get that \re{5.31} holds.\end{proof}

\vskip 2mm

At last,  we give an estimate for $\rho_{n}(t)(n\ge2)$.

Define
\begin{align}\label{4.44}
\vartheta_{n}=\frac{\int_{0}^{T}\frac{\gamma n}{R^{3}_{*}(t)}\left(\frac{n(n+1)}{2}-1\right)dt}{\int_{0}^{T} \Phi(t)  R_{*}^{2}(t) P_{0}(R_{*}(t)){[P_{1}(R_{*}(t))-P_{n}(R_{*}(t))]}dt} \qquad n\ge2.
\end{align}

\begin{lem}\label{lem:1}
  For $n\geqslant 2$, $ \vartheta_{n} <\vartheta_{n+1}$.
\end{lem}
\begin{proof}
From \eqref{4.44}, we have
\begin{align}\nonumber
\vartheta_{n}=\frac{\int_{0}^{T}\frac{\gamma}{R^{3}_{*}(t)}dt}{\int_{0}^{T} \Phi(t)  R_{*}^{2}(t) P_{0}(R_{*}(t))\frac{[P_{1}(R_{*}(t))-P_{n}(R_{*}(t))]}{n\left(\frac{n(n+1)}{2}-1\right)}dt} \qquad n\ge2.
\end{align}
 Since the sequence $ \frac{  n\Big(n(n+1)/2-1\Big)
}{   P_1(R)-P_n(R)   }$ is strictly increasing in $n$ as $n\geqslant2$ (\cite{Friedman2001}), we get that this lemma is true.
\end{proof}

Lemma \ref{lem:5.2} and \ref{lem:5.3} hold for all $\mu>0$, so we define $\vartheta_{0}=\vartheta_{1}=\infty$. Set
\begin{equation}
\mu_{*}=\min \left\{\vartheta_{0}, \vartheta_{1}, \vartheta_{2}, \vartheta_{3}, \vartheta_{4}, \cdots\right\}.
\end{equation}
From Lemma \ref{lem:1}, we obtain that $\mu_{*}=\vartheta_{2}$.

\begin{lem}\label{thm3q.3}
For $n \ge 2$ and $0<\mu<\vartheta_{2}$, there exist $\delta>0$ and $M>0$ such that
\begin{equation}\nonumber
\left|\rho_{n, m}(t)\right|\leqslant\left|\rho_{n, m}(0)\right| e^{-\delta\left(n^{3}+1\right) t} \qquad t>M,
\end{equation}
where $\delta$ and $M$ are independent of $n$ and depend on $R_{*}(t), \Phi(t) ,T, \mu$ and $\gamma$.
\end{lem}
\begin{proof}
  For any $t > T$, there exist a positive integer $m$ and $\tau \in[0, T)$ such that $t=m T+\tau$.
From \eqref{2.8}, we have
\begin{align}\nonumber
 -\int_{0}^{\tau}&\!\frac{\gamma n}{R^{3}_{*}(t)}\Big(\frac{n(n+1)}{2}\!-\!1\Big)\!-\!\mu \Phi(t) R_{*}^{2}(t) P_{0}(R_{*}(t))[P_{1}(R_{*}(t))\!-\!P_{n}(R_{*}(t))]dt\\\nonumber
& \le  \int_{0}^{\tau}\mu \Phi(t) R_{*}^{2}(t) P_{0}(R_{*}(t))P_{1}(R_{*}(t))dt \\\label{5.36}
&\le \frac{1}{15}\vartheta_{2}\Phi^{\ast}R_{max}^{2}T ,
\end{align}
where $R_{max}=\max\limits_{t>0}\left\{R_{\ast}(t)\right\}$.
From \re{4.44}, we obtain
\begin{align}\nonumber
 \int_{0}^{T}&\!\frac{\gamma n}{R^{3}_{*}(t)}\Big(\frac{n(n+1)}{2}\!-\!1\Big)\!-\!\mu \Phi(t) R_{*}^{2}(t) P_{0}(R_{*}(t))[P_{1}(R_{*}(t))\!-\!P_{n}(R_{*}(t))]dt\\\nonumber
&= \frac{\mu}{\vartheta_{2}}  \int_{0}^{T}\frac{\vartheta_{2}}{\mu}\frac{\gamma n}{R^{3}_{*}(t)}\left(\frac{n(n+1)}{2}-1\right)-\vartheta_{2} \Phi (t) R_{*}^{2} (t) P_{0}(R_{*}(t))[P_{1}(R_{*}(t))\!-\!P_{n}(R_{*}(t))] dt\\\nonumber
&\geqslant  \frac{\mu}{\vartheta_{2}}(\frac{\vartheta_{2}}{\mu}-1) \int_{0}^{T}\frac{\gamma n}{R^{3}_{*}(t)}\left(\frac{n(n+1)}{2}-1\right)dt\\\label{5.231}
&\geqslant   \frac{\mu}{\vartheta_{2}}(\frac{\vartheta_{2}}{\mu}-1)\frac{\gamma}{R_{max}^3}\frac{n^{3}+1}{4}  T =M_{1}(n^{3}+1)T,
\end{align}
where 
$M_1=\frac{1}{4}\frac{\mu}{\vartheta_{2}}(\frac{\vartheta_{2}}{\mu}-1)  \frac{\gamma}{R_{max}^3}>0$.

The fact that $R_{*}(t)$ and $\Phi(t) $ are T-periodic, \re{5.26}, \re{5.36} and \re{5.231} imply
\begin{align}\nonumber
&\left|\rho_{n}(t)\right|\\\nonumber
&=\left|\rho_{n}(0)\right|\frac{R^{n-1}_{*}(0)}{R^{n-1}_{*}(t)}\exp\Big\{-\int_{0}^{mT+\tau}\!\frac{\gamma n}{R^{3}_{*}(t)}\Big(\frac{n(n+1)}{2}\!-\!1\Big)\!-\!\mu \Phi(t) R_{*}^{2}(t) P_{0}(R_{*}(t))[P_{1}(R_{*}(t))\!-\!P_{n}(R_{*}(t))]dt\Big\}\\\nonumber
&\leqslant\left|\rho_{n}(0)\right|\frac{R^{n-1}_{*}(0)}{R^{n-1}_{*}(\tau)}\exp\Big\{- m\int_{0}^{T}\!\frac{\gamma n}{R^{3}_{*}(t)}\Big(\frac{n(n+1)}{2}\!-\!1\Big)\!-\!\mu \Phi(t) R_{*}^{2}(t) P_{0}(R_{*}(t))[P_{1}(R_{*}(t))\!-\!P_{n}(R_{*}(t))]dt\Big\}\\\nonumber
&
\qquad \qquad  \qquad \exp\Big\{-\int_{0}^{\tau}\!\frac{\gamma n}{R^{3}_{*}(t)}\Big(\frac{n(n+1)}{2}\!-\!1\Big)\!-\!\mu \Phi(t) R_{*}^{2}(t) P_{0}(R_{*}(t))[P_{1}(R_{*}(t))\!-\!P_{n}(R_{*}(t))]dt\Big\}\\\nonumber
& \le \left|\rho_{n}(0)\right|\frac{R^{n-1}_{max}}{R^{n-1}_{min}}\exp\Big\{\frac{1}{15}\vartheta_{2}\Phi^{\ast}R_{max}^{2}T \Big\}\exp\Big\{-M_{1}(n^{3}+1)mT\Big\}\\\nonumber
& \le \left|\rho_{n}(0)\right|\frac{R^{n-1}_{max}}{R^{n-1}_{min}}\exp\Big\{\frac{1}{15}\vartheta_{2}\Phi^{\ast}R_{max}^{2}T \Big\}\exp\Big\{-M_{1}(n^{3}+1)(t-T)\Big\}.
\end{align}


Then there exists $M>0$ such that
\begin{equation}\nonumber
\left|\rho_{n}(t)\right|\leqslant\left|\rho_{n}(0)\right|e^{-\frac{M_{1}}{2}(n^{3}+1)t} \qquad t>M.
\end{equation}
Taking $\delta=\frac{M_{1}}{2}$, we complete the proof.
\end{proof}

\vskip 3mm\noindent

Our main result for three--space dimensional problem is the following theorem.
\begin{thm}\label{thm:5.1}
Assume $\rho_{0} \in L^{2}(\partial B_{R_{*}(0)})$. Then\\
(i) if $\mu\in(0, \vartheta_{2})$, then the radially symmetric T-periodic solution $\left(\sigma_{*}(r, t), p_{*}(r, t), R_{*}(t)\right)$ (obtained in Corollary \ref{cor:1.3333}) is linearly stable, i.e., for any positive integer $k$, there exist $\d>0$, $C>0$ and $t_0>0$ such that
\begin{align}\nonumber
 \|\rho( \theta,\phi, t)-\sum_{m=-1}^1 \rho_{1,m}(0) Y_{1,m}\|_{H^{k+\frac{1}{2}}(\partial B_{R_{*}(t)})}\le Ce^{-\d t}\quad
  \for t>t_0.
 \end{align}
(ii) If $\mu> \vartheta_{2}$, then the radially symmetric T-periodic solution $\left(\sigma_{*}(r, t), p_{*}(r, t), R_{*}(t)\right)$  is linearly unstable.
\end{thm}

\begin{proof}
For any positive integer $k$, Lemma \ref{lem:5.3}, \cite[Lemma 8.2]{Friedman2001a}, Lemma \ref{lem:5.2} and Lemma \ref{thm3q.3} imply that there exist $\delta>0$, $C>0$ and $M>0$ such that
\begin{align}\nonumber
\Big\|&\rho( \theta,\phi, t)-\sum_{m=-1}^1 \rho_{1,m}(0) Y_{1,m}\Big\|_{H^{k+\frac{1}{2}}(\partial B_{R_{*}(t)})}^{2}\\\nonumber
&=\Big\|\sum_{n=0,n\neq1}^\infty\sum_{m=-n}^{n}\rho_{n,m}(t) Y_{n,m}(\theta,\phi)\Big\|_{H^{k+\frac{1}{2}}(\partial B_{R_{*}(t)})}^{2}\\\nonumber
& \le \sum_{n=0, n\neq 1}^\infty \sum_{m=-n}^{n}(1+n^{2k+1}) |\rho_{n,m}(t) |^2\\
& \le \sum_{n=0, n\neq 1}^\infty \sum_{m=-n}^{n}(1+n^{2k+1}) \left|\rho_{n, m}(0)\right|^2 e^{-2\delta\left(n^{3}+1\right) t}\nonumber
 \le Ce^{-\delta t} \qquad for~ t>M.
 \end{align}
%
%

We next turn to linear instability for $\mu>\vartheta_{2}$. For any $t > T$, there exist a positive integer $m$ and $\tau \in[0, T)$ such that $t=m T+\tau$.  By \re{2.10}, we derive
\begin{align}\nonumber
\int_{0}^{\tau}\!&-\frac{4\gamma}{R^{3}_{*}(t)}\!+\!\mu \Phi(t) R_{*}^{2}(t) P_{0}(R_{*}(t))[P_{1}(R_{*}(t))\!-\!P_{2}(R_{*}(t))]dt\\\label{4.55}
&\geqslant- \int_{0}^{\tau}\!\frac{4\gamma}{R^{3}_{*}(t)} dt \geqslant-\frac{4\gamma}{R^{3}_{min}}T,
\end{align}
where $R_{min}=\min\limits_{t>0}\left\{R_{\ast}(t)\right\}$.
From \re{4.44}, we obtain
\begin{align}\nonumber
\int_{0}^{T}\!&-\frac{4\gamma}{R^{3}_{*}(t)}\!+\!\mu \Phi(t) R_{*}^{2}(t) P_{0}(R_{*}(t))[P_{1}(R_{*}(t))\!-\!P_{2}(R_{*}(t))]dt\\\nonumber
& = \frac{\mu}{\vartheta_{2}}  \int_{0}^{T}\!-\frac{\vartheta_{2}}{\mu}\frac{4\gamma}{R^{3}_{*}(t)}\!+\!\vartheta_{2}\Phi(t) R_{*}^{2}(t) P_{0}(R_{*}(t))[P_{1}(R_{*}(t))\!-\!P_{2}(R_{*}(t))]dt\\\nonumber
&=\frac{\mu}{\vartheta_{2}}(1-\frac{\vartheta_{2}}{\mu}) \int_{0}^{T}\!\frac{4\gamma}{R^{3}_{*}(t)}dt\\\label{4.56}
&\geqslant \frac{\mu}{\vartheta_{2}}(1-\frac{\vartheta_{2}}{\mu}) \frac{4\gamma}{R^{3}_{max}}T,
\end{align}
where $R_{max}=\max\limits_{t>0}\left\{R_{\ast}(t)\right\}$.

The fact that $R_{*}(t)$ and $\Phi(t) $ are T-periodic, \re{5.26}, \re{4.55} and \re{4.56} imply
\begin{align}\nonumber
\left|\rho_{2, m}(t)\right|
&=\left|\rho_{2, m}(0)\right|\!\exp\! \Big\{\!\!-\!\!\int_{0}^{t}\frac{d R_{*}}{d t}\frac{1}{R_{*}(t)}\!+\!\frac{4\gamma}{R^{3}_{*}(t)}\!-\!\mu \Phi(t) R_{*}^{2}(t) P_{0}(R_{*}(t))[P_{1}(R_{*}(t))\!-\!P_{2}(R_{*}(t))]dt\Big\}\\\nonumber
& =\left|\rho_{2, m}(0)\right|\frac{R_{*}(0)}{R_{*}(t)}\!\exp\! \Big\{\int_{0}^{t}\!-\frac{4\gamma}{R^{3}_{*}(t)}\!+\!\mu \Phi(t) R_{*}^{2}(t) P_{0}(R_{*}(t))[P_{1}(R_{*}(t))\!-\!P_{2}(R_{*}(t))]dt\Big\}\\\nonumber
& =\left|\rho_{2, m}(0)\right|\frac{R_{*}(0)}{R_{*}(\tau)}\!\exp\! \Big\{m\int_{0}^{T}\!-\frac{4\gamma}{R^{3}_{*}(t)}\!+\!\mu \Phi(t) R_{*}^{2}(t) P_{0}(R_{*}(t))[P_{1}(R_{*}(t))\!-\!P_{2}(R_{*}(t))]dt\Big\}\\\nonumber
&\quad 
 \qquad \qquad \qquad \exp \Big\{\int_{0}^{\tau}\!-\frac{4\gamma}{R^{3}_{*}(t)}\!+\!\mu \Phi(t) R_{*}^{2}(t) P_{0}(R_{*}(t))[P_{1}(R_{*}(t))\!-\!P_{2}(R_{*}(t))]dt\Big\}\\\nonumber
& \geqslant\left|\rho_{2, m}(0)\right| \frac{R_{min}}{R_{max}}\exp\Big\{-\frac{4\gamma}{R^{3}_{min}}T \Big\}\exp\Big\{\frac{\mu}{\vartheta_{2}}(1-\frac{\vartheta_{2}}{\mu}) \frac{4\gamma}{R^{3}_{max}}mT\Big\}\\\nonumber
& \geqslant\left|\rho_{2, m}(0)\right| \frac{R_{min}}{R_{max}}\exp\Big\{-\frac{4\gamma}{R^{3}_{min}}T \Big\}\exp\Big\{\frac{\mu}{\vartheta_{2}}(1-\frac{\vartheta_{2}}{\mu}) \frac{4\gamma}{R^{3}_{max}}(t-T)\Big\}.
\end{align}
Then
\begin{align*}
  \left|\rho_{2,m}(t)\right| \rightarrow \infty \qquad  t \rightarrow \infty.
\end{align*}
Therefore, the proof is completed.
\end{proof}
\section*{Acknowledgements}
The research was supported by NNSF of China (11471339) and NSF of Guangdong (2018A030313523).
\bibliographystyle{elsarticle-num}
\bibliography{2(2)}

\end{document}